\documentclass[11pt]{article}
\usepackage[margin=1in]{geometry} 
\usepackage{amssymb,amsfonts,amsmath,bbm,mathrsfs,stmaryrd}
\usepackage{xcolor}
\usepackage{url}

\usepackage[colorlinks,
             linkcolor=black!75!red,
             citecolor=blue,
             pdftitle={Park},
             pdfauthor={Alexandru Chirvasitu},
             pdfproducer={pdfLaTeX},
             pdfpagemode=None,
             bookmarksopen=true
             bookmarksnumbered=true]{hyperref}

\usepackage{tikz}
\usetikzlibrary{arrows,calc,decorations.pathreplacing,decorations.markings,intersections,shapes.geometric,through,fit,shapes.symbols,positioning,decorations.pathmorphing}

\usepackage[amsmath,thmmarks,hyperref]{ntheorem}
\usepackage{cleveref}

\crefname{section}{Section}{Sections}
\crefformat{section}{#2Section~#1#3} 
\Crefformat{section}{#2Section~#1#3} 

\crefname{subsection}{\S}{\S\S}
\crefformat{subsection}{#2\S#1#3} 
\Crefformat{subsection}{#2\S#1#3} 

\theoremstyle{plain}

\newtheorem{lemma}{Lemma}[section]
\newtheorem{proposition}[lemma]{Proposition}
\newtheorem{corollary}[lemma]{Corollary}
\newtheorem{theorem}[lemma]{Theorem}

\newtheorem{lemma'}{Lemma}%[section]
\newtheorem{proposition'}[lemma']{Proposition}
\newtheorem{corollary'}[lemma']{Corollary}
\newtheorem{theorem'}[lemma']{Theorem}
\newtheorem{conjecture'}[lemma']{Conjecture}
\newtheorem{question'}[lemma']{Question}
\newtheorem{assumption'}[lemma']{Assumption}

\theoremstyle{nonumberplain}

\theoremstyle{plain}
\theorembodyfont{\upshape}
\theoremsymbol{\ensuremath{\blacklozenge}}

\newtheorem{definition}[lemma]{Definition}
\newtheorem{notation}[lemma]{Notation}

\newtheorem{remark}[lemma]{Remark}

\newtheorem{definition'}[lemma']{Definition}
\newtheorem{notation'}[lemma']{Notation}
\newtheorem{example'}[lemma']{Example}
\newtheorem{remark'}[lemma']{Remark}
\newtheorem{convention'}[lemma']{Convention}

\crefname{definition}{definition}{definitions}
\crefformat{definition}{#2definition~#1#3} 
\Crefformat{definition}{#2Definition~#1#3} 

\crefname{definition'}{definition}{definitions}
\crefformat{definition'}{#2definition~#1#3} 
\Crefformat{definition'}{#2Definition~#1#3}

\crefname{ex}{example}{examples}
\crefformat{example}{#2example~#1#3} 
\Crefformat{example}{#2Example~#1#3} 

\crefname{remark}{remark}{remarks}
\crefformat{remark}{#2remark~#1#3} 
\Crefformat{remark}{#2Remark~#1#3} 

\crefname{convention}{convention}{conventions}
\crefformat{convention}{#2convention~#1#3} 
\Crefformat{convention}{#2Convention~#1#3}

\crefname{lemma}{lemma}{lemmas}
\crefformat{lemma}{#2lemma~#1#3} 
\Crefformat{lemma}{#2Lemma~#1#3} 

\crefname{proposition}{proposition}{propositions}
\crefformat{proposition}{#2proposition~#1#3} 
\Crefformat{proposition}{#2Proposition~#1#3} 

\crefname{proposition'}{proposition}{propositions}
\crefformat{proposition'}{#2proposition~#1#3} 
\Crefformat{proposition'}{#2Proposition~#1#3}

\crefname{corollary}{corollary}{corollaries}
\crefformat{corollary}{#2corollary~#1#3} 
\Crefformat{corollary}{#2Corollary~#1#3}

\crefname{theorem}{theorem}{theorems}
\crefformat{theorem}{#2theorem~#1#3} 
\Crefformat{theorem}{#2Theorem~#1#3}

\crefname{theorem'}{theorem}{theorems}
\crefformat{theorem'}{#2theorem~#1#3} 
\Crefformat{theorem'}{#2Theorem~#1#3}

\crefname{assumption}{assumption}{Assumptions}
\crefformat{assumption}{#2assumption~#1#3} 
\Crefformat{assumption}{#2Assumption~#1#3}

\crefname{question}{question}{Questions}
\crefformat{question}{#2question~#1#3} 
\Crefformat{question}{#2Question~#1#3} 

\crefname{question'}{question}{Questions}
\crefformat{question'}{#2question~#1#3} 
\Crefformat{question'}{#2Question~#1#3}

\crefname{equation}{}{}
\crefformat{equation}{(#2#1#3)} 
\Crefformat{equation}{(#2#1#3)}

\theoremstyle{nonumberplain}
\theoremsymbol{\ensuremath{\blacksquare}}

\newtheorem{proof}{Proof}

\newtheorem{proof_of_symbis}{Proof of \Cref{th.sym_bis}}
\newtheorem{proof_of_ellbis}{Proof of \Cref{th.ell_bis}}
\newtheorem{proof_of_bla}{Proof of \Cref{pr.bla}}
\newtheorem{proof_of_aux_bis}{Proof of \Cref{pr.aux_bis}}
\newtheorem{proof_of_illum}{Proof of \Cref{th.illum}}

\newcommand\bR{{\mathbb R}}
\newcommand\bS{{\mathbb S}}

\newcommand\cG{{\mathcal G}}

\DeclareMathOperator{\co}{\mathrm{co}}

\newcommand{\define}[1]{{\em #1}}

\newcommand{\qedhere}{\mbox{}\hfill\ensuremath{\blacksquare}}

\title{Parkable convex sets and finite-dimensional Hilbert spaces}
\author{Alexandru Chirvasitu\footnote{University of Washington, \url{chirva@uw.edu}}}

\begin{document}

\date{}

\maketitle

\begin{abstract}
A subset of a convex body $B$ containing the origin in a Euclidean space is {\it parkable in $B$} if it can be translated inside $B$ in such a manner that the translate the origin. We provide characterizations of ellipsoids and of centrally symmetric convex bodies in Euclidean spaces of dimension $\ge 3$ based on the notion of parkability, answering several questions posed by G. Bergman. 

The techniques used, which are based on characterizations of Hilbert spaces among finite-dimensional Banach spaces in terms of their lattices of subspaces and algebras of endomorphisms, also apply to improve a result of W. Blaschke characterizing ellipsoids in terms of boundaries of illumination. 
\end{abstract}


\vspace{.5cm}


%\tableofcontents

%%%%%%%%%%%%%%%%%%%%%%%%%%%%%%%%%%%%%%%%%%%%%%%%%%%%%%%%%%%%%%%%%%%%%%%%%%%%%%%%%%%%%%%%%%%%%%%%%%%%%%%%%%%%%%%%%%
%%%%%%%%%%%%%%%%%%%%%%%%%%%%%%%%%%%%%%%%%%%%%%%%%%%%%%%%%%%%%%%%%%%%%%%%%%%%%%%%%%%%%%%%%%%%%%%%%%%%%%%%%%%%%%%%%%
\section*{Introduction}

The present paper answers several questions raised in \cite{berg} regarding convex bodies in euclidean spaces. The setup is based on the following notion; it is introduced by G. Bergman in the process of studying ``efficient'' embeddings of metric spaces into other metric spaces. 

All convex sets are understood to be subsets of some ambient Euclidean space $\bR^n$.

\begin{definition'}\label{def.parkable}
  Let $C\subseteq B$ be convex sets with $0\in B$. $C$ is \define{parkable in} $B$ if some translate of $C$ is still contained in $B$ and contains $0$. 
\end{definition'}

The questions in \cite{berg} referred to above have to do with characterizing particularly nice convex bodies in $\bR^n$ (i.e. centrally symmetric or ellipsoids) by means of parkability. First recall

\begin{definition'}
  Let $C\subseteq \bR^n$ be a convex subset. A \define{center of symmetry} for $C$ is a point $p\in C$ such that
  \begin{equation*}
    C=2p-C:=\{2p-x\ |\ x\in C\}. 
  \end{equation*}
  $C$ \define{has a center of symmetry} if a center of symmetry exists. 

  $C$ is \define{centrally symmetric} if $0\in C$ is a center of symmetry for it. 
\end{definition'}

\cite[Question 32]{berg} then reads as follows.

\begin{question'}\label{qu.sym}
  Let $C$ be a compact convex subset of $\bR^n$ for some $n>2$, with the property that for every centrally symmetric compact convex subset $B\subset \bR^n$ containing some translate of $C$, that translate is parkable in $B$. 

Is it true that $C$ must have a center of symmetry?  
\end{question'}

One of the main results is

\begin{theorem'}\label{th.sym}
  The answer to \Cref{qu.sym} is affirmative. 
\qedhere
\end{theorem'}

A more elaborate result has to do with recognizing ellipsoids by means of parkability. Recall

\begin{definition'}
  A \define{convex body} is a compact convex set with non-empty interior which contains $0$. 
\end{definition'}

The starting point for the discussion that follows is the following result from \cite{berg}.

\begin{proposition'}\label{pr.ell}
 Let $B\subset \bR^n$ be a convex body for some $n>2$. Then, the
 following properties are successively weaker. 
 \begin{enumerate}
 \renewcommand{\labelenumi}{(\roman{enumi})}  
   \item $B$ is an ellipsoid centered at $0$.  
   \item $B$ is centrally symmetric and its intersection with every hyperplane has a center of symmetry (provided it is not empty). 
   \item Every closed convex subset of $B$ is parkable in $B$. 
 \end{enumerate}
\end{proposition'}

It is then natural to ask \cite[Question 31]{berg}:

\begin{question'}\label{qu.ell}
  Is either of the implications from \Cref{pr.ell} reversible?
\end{question'}

In this context, the second main result is

\begin{theorem'}\label{th.ell}
  The three properties in \Cref{pr.ell} are equivalent to one another. 
  \qedhere
\end{theorem'}

The rest of this introduction contains a few remarks about the proofs.

\Cref{th.sym,th.ell} do not directly imply one another as such, but they are nevertheless interlinked through the methods used in their proofs. The same auxiliary lemma about plane compact convex sets, for instance, leads both to \Cref{th.sym} and to the fact that condition (iii) in \Cref{pr.ell} implies that $B$ is centrally symmetric (in other words, the partial implication (iii) $\Rightarrow$ (ii)).

Once we show that (iii) implies central symmetry, we can assume $B$ to be centrally symmetric throughout the rest of the proof of \Cref{th.ell}. As such, it is the unit ball of a unique Banach space structure $(\bR^n,\|\cdot\|)$ on $\bR^n$, and now functional-analytic techniques and results can be brought to bear.

More specifically, using the same auxiliary result referred to in passing above we first prove

\begin{proposition'}\label{pr.aux}
  Let $B$ be a convex centrally symmetric body $B$ satisfying condition (iii) from \Cref{pr.ell}. Then, for every linear hyperplane $H\subset \bR^n$, the non-empty intersections
  \begin{equation*}
    (H+x)\cap B,\quad x\in \bR^n
  \end{equation*}
have centers of symmetry. Moreover, these centers are collinear. \qedhere
\end{proposition'}

Associating to $L$ the unique line through the origin containing the centers of symmetry from the statement of \Cref{pr.aux} extends, it turns out, to an inclusion-reversing involution on the lattice of subspaces of the Banach space $(\bR^n,\|\cdot\|)$ referred to above. This, together with a lattice-theoretic characterization of Hilbert spaces among Banach spaces due to Kakutani et al. leads to the conclusion that $(\bR^n,\|\cdot\|)$ is a Hilbert space. But this is equivalent to its unit ball being an ellipsoid, and the conclusion follows.

There are other results that might be of independent interest, such as an improvement of a result of Blaschke on characterizations of ellipsoids by means of light rays (\cite{bla}).

%%%%%%%%%%%%%%%%%%%%%%%%%%%%%%%%%%%%%%%%%%%%%%%%%%%%%%%%%%%%%%%%%%%%%%%%%%%%%%%%%%%%%%%%%%%%%%%%%%%%%%%%%%%%%%%%%%
%\subsection*{Acknowledgements}

%%%%%%%%%%%%%%%%%%%%%%%%%%%%%%%%%%%%%%%%%%%%%%%%%%%%%%%%%%%%%%%%%%%%%%%%%%%%%%%%%%%%%%%%%%%%%%%%%%%%%%%%%%%%%%%%%%
%%%%%%%%%%%%%%%%%%%%%%%%%%%%%%%%%%%%%%%%%%%%%%%%%%%%%%%%%%%%%%%%%%%%%%%%%%%%%%%%%%%%%%%%%%%%%%%%%%%%%%%%%%%%%%%%%%
\section{Preliminaries}\label{se.prel}

%%%%%%%%%%%%%%%%%%%%%%%%%%%%%%%%%%%%%%%%%%%%%%%%%%%%%%%%%%%%%%%%%%%%%%%%%%%%%%%%%%%%%%%%%%%%%%%%%%%%%%%%%%%%%%%%%%
\subsection{Convex geometry}\label{subse.conv}

Our main reference on the topic will be \cite{tomo}, to which we refer the reader for basic terminology. 

We say that two convex bodies in a Euclidean space $\bR^n$ are {\it mutual translates} if one is the image of the other through some translation of $\bR^n$.

The following result will make an appearance several times in the
sequel; the reader can consult e.g. \cite{Rya15} and the references
therein for background on the result (which is also \cite[Theorem 3.1.3]{tomo}).

\begin{theorem}\label{th.transl}
 Let $K$ and $L$ be two compact convex subsets of $\bR^n$ for $n\ge 3$, and $2\le m\le {n-1}$ a positive integer. If the projections of $K$ and $L$ on every $m$-dimensional linear subspace of $\bR^n$ are mutual translates, then so are $K$ and $L$. 
\qedhere 
\end{theorem}

As a consequence, we get (\cite[Corollary 3.1.5]{tomo})

\begin{corollary}\label{cor.centr}
  Let $m$ and $n$ be positive integers as in \Cref{th.transl}, and
  $B\subset \bR^n$ a compact convex subset. Then, $B$ has a center of symmetry
  if and only if its projection on every $m$-dimensional linear
  subspace of $\bR^n$ does. 
\end{corollary}
\begin{proof}
  Having a center of symmetry is equivalent to $B$ and $-B$ being
  translates; we can now simply apply \Cref{th.transl} to this pair of
  convex sets. 
\end{proof}

%%%%%%%%%%%%%%%%%%%%%%%%%%%%%%%%%%%%%%%%%%%%%%%%%%%%%%%%%%%%%%%%%%%%%%%%%%%%%%%%%%%%%%%%%%%%%%%%%%%%%%%%%%%%%%%%%%
\subsection{Banach and Hilbert spaces}\label{subse.ban}

We will assume some basics on Banach and Hilbert spaces and bounded operators thereon, as covered e.g. in the introductory sections of \cite[Chapters 1 and 12]{rud} or in the first three chapters of \cite{con}.

Given a centrally symmetric convex body $B=-B$ in $\bR^n$, we can associate to it the unique Banach space structure on $\bR^n$ making $B$ the unit ball: the norm of $x\in \bR^n$ is defined to be
\begin{equation*}
  \|x\| = \|x\|_B = \inf\{r\ge 0\ |\ x\in rB\}. 
\end{equation*}
The correspondence $B\mapsto \|\cdot\|_B$ is a bijection between centrally symmetric convex bodies and Banach space structures on $\bR^n$; the inverse map associates to a Banach space structure $(\bR^n,\|\cdot\|)$ its unit ball.

As \Cref{th.ell} above suggests, one of our main goals will be characterizing ellipsoids among convex bodies. In terms of the bijection $B\leftrightarrow \|\cdot\|_B$ the ellipsoids correspond to the {\it Hilbert space} structures on $\bR^n$, i.e. those Banach space structures whose underlying norm $\|\cdot\|$ arises from an inner product $\langle-,-\rangle$ via the usual formula
\begin{equation*}
  \|x\|^2 = \langle x,x\rangle,\ \forall x\in \bR^n. 
\end{equation*}
For this reason, it will be important to have at our disposal results that allow for the recognition of Hilbert spaces among Banach spaces. One such tool is \cite[Theorem 1.1]{istr} (or rather a variant thereof, with real Banach spaces instead of complex ones):

\begin{theorem}\label{th.istr} 
  Let $(\bR^n,\|\cdot\|)$ be a Banach space. The norm is induced by an inner product if and only if there exists an operation $T\mapsto T^*$ on the space $M_n(\bR)$ of endomorphisms of $\bR^n$ such that
  \begin{enumerate}
    \item $(T+S)^* = T^*+S^*$; 
    \item $(T^*)^* = T$;
    \item $(TS)^* = S^*T^*$;
    \item $\|P\|\le 1$ if $P^2=P=P^*$, where $\|P\|$ is the norm on $M_n(\bR)$ induced by that on $\bR^n$. 
  \end{enumerate}
\qedhere
\end{theorem}

In other words, we have a recognition criterion for Hilbert space norms in terms of involutions $T\mapsto T^*$ on their algebras of endomorphisms.

%%%%%%%%%%%%%%%%%%%%%%%%%%%%%%%%%%%%%%%%%%%%%%%%%%%%%%%%%%%%%%%%%%%%%%%%%%%%%%%%%%%%%%%%%%%%%%%%%%%%%%%%%%%%%%%%%%
%%%%%%%%%%%%%%%%%%%%%%%%%%%%%%%%%%%%%%%%%%%%%%%%%%%%%%%%%%%%%%%%%%%%%%%%%%%%%%%%%%%%%%%%%%%%%%%%%%%%%%%%%%%%%%%%%%
\section{Universally parkable sets}\label{se.univ}

The aim of this section is to prove \Cref{th.sym} above, answering \Cref{qu.sym} in the affirmative. We recall the statement, after introducing the following notion relevant to the setup of the theorem.

\begin{definition}\label{def.univ_park}
  A compact convex set $C\subset \bR^n$ is {\it universally parkable} if for any centrally symmetric convex subset $B\subset \bR^n$ containing a translate of $C$, that translate is parkable in $B$ in the sense of \Cref{def.parkable}. 
\end{definition}

\begin{theorem}\label{th.sym_bis}
  For a positive integer $n\ge 3$, every universally parkable compact
  convex subset $C\subset \bR^n$ has a center of symmetry.
\end{theorem}

We proceed through a series of auxiliary results. First off, we narrow down the class of centrally symmetric convex sets that witness the parkability of $C\subset \bR^n$. For this purpose as well as for use in the sequel, denote by $C_u$ the translate $C+u$ for $u\in \bR^n$.

\begin{lemma}\label{le.Cu}
  A compact convex subset $C\subset \bR^n$ is universally parkable if and only if for every $u\in \bR^n$ the translate $C_u$ is parkable in the convex hull $\co(-C_u\bigcup C_u)$.
\end{lemma}
\begin{proof}
  On the one hand, if $C$ is universally parkable, every $\co(-C_u\bigcup C_u)$ is certainly a centrally symmetric convex subset of $\bR^n$ containing a translate $C_u$, and hence the latter must be parkable in the former. 

Conversely, every centrally symmetric compact convex $B\subset \bR^n$ containing a translate $C_u$ contains $\co(-C_u\bigcup C_u)$, so that if $C_u$ is parkable in the latter then it is parkable in the former as well. 
\end{proof}

As a consequence, we obtain

\begin{corollary}\label{cor.Cu}
  If $C\subset \bR^n$ is universally parkable, then so are its projections on linear subspaces of $\bR^n$. 
\end{corollary}
\begin{proof}
  This follows from the fact that the property from the statement of \Cref{le.Cu} is clearly invariant under taking orthogonal projections. 
\end{proof}

In order to fix ideas, we now specialize \Cref{th.sym_bis} to $n=3$.

\begin{lemma}
  If \Cref{th.sym_bis} holds for $n=3$ then it holds in general. 
\end{lemma}
\begin{proof}
  Starting with a universally parkable $C\subset \bR^n$ for some $n\ge 3$, \Cref{cor.Cu} ensures that its projections on all $3$-dimensional linear subspaces of $\bR^n$ are universally parkable. If \Cref{th.sym_bis} holds for $\bR^3$ we can conclude that these projections all have centers of symmetry, which according to \Cref{cor.centr} means that so does $C$. 
\end{proof}

This now allows us to focus on the case $n=3$. Let $\bS\subset \bR^3$ be an origin-centered sphere whose radius
is large enough to ensure that no translates $C_u$ (and hence no reflections
$-C_u$) contain $0$ as $u$ ranges over $\bS$. The following lemma shows
that when $u\in \bS$ there is a certain rigidity in how one may
translate $C_u$ inside $\co(-C_u\bigcup C_u)$ so as to engulf $0$.

\begin{lemma}\label{le.unique}
  Let $C\subset \bR^3$ be as in the statement of \Cref{th.sym_bis} For
  any $u\in \bS$, the direction of a vector $0\ne v\in \bR^3$ such that 
  \begin{equation*}
    0\in C_u+v\subset \co(-C_u\bigcup C_u)
  \end{equation*}
 is uniquely determined. 
\end{lemma}
\begin{proof}
  Let $u\in \bS$ and $0\ne v\in \bR^3$ be vectors as in the statement, and let
  $H\subset \bR^3$ be a $2$-dimensional linear such that the
  projection $C_u|H$ does not contain $0\in H$ and is not collinear
  with the origin (i.e. is not a segment whose supporting line
  contains $0$). Note that such planes
  $H$ form an open subset of the Grassmannian $\cG(3,2)$ (because
  their defining property is open), and by our
  choice of radius for $\bS$ this subset of $\cG(3,2)$ is non-empty. 

In the plane $H$ the part of the boundary $\partial \co(-C_u|H\bigcup
C_u|H)$ that is not contained in the boundaries of $C_u|H$ and
$-C_u|H$ consists of two open segments forming two opposite edges of a
parallelogram centered at $0$:
  
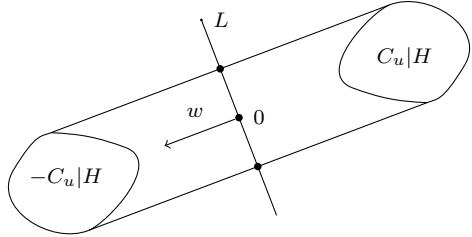
\begin{equation*}
      \begin{tikzpicture}[auto,baseline=(current  bounding
        box.center),scale=1]
        \node (c) at (2.2,.8) {$\scriptstyle C_u|H$};
        \node (-c) at (-2.3,-.8) {$\scriptstyle -C_u|H$};
        \node (aux) at (-1.125,-.425) {};
        \draw (0,0) node[circle, inner sep=1pt, fill=black,
        label={right:{$\scriptstyle 0$}}] (0) {};
%%%%%%%%%%%%%%%%%%%%%%%%%%%%%%%%%%%%%%%%%%%%%%%
        \draw[-] (2,1.5) .. controls (2.5,1.7) and (3.3,1.2) .. (3,.7);
        \draw[-] (3,.7) .. controls (2.7,.2) and (2.625,.25) .. (2.5,.2);
        \draw[-] (2.5,.2) to[bend left=9] (1.5,.4);
        \draw[-] (1.5,.4) .. controls (1,.6) and (1.75,1.4) .. (2,1.5);
%%%%%%%%%%%%%%%%%%%%%%%%%%%%%%%%%%%%%%%%%%%%%%%
        \draw[-] (-2,-1.5) .. controls (-2.5,-1.7) and (-3.3,-1.2) .. (-3,-.7);
        \draw[-] (-3,-.7) .. controls (-2.7,-.2) and (-2.625,-.25) .. (-2.5,-.2);
        \draw[-] (-2.5,-.2) to[bend left=9] (-1.5,-.4);
        \draw[-] (-1.5,-.4) .. controls (-1,-.6) and (-1.75,-1.4) .. (-2,-1.5);
%%%%%%%%%%%%%%%%%%%%%%%%%%%%%%%%%%%%%%%%%%%%%%%
        \draw[-] (2,1.5) -- (-2.5,-.2);
        \draw[-] (-2,-1.5) -- (2.5,.2);
%%%%%%%%%%%%%%%%%%%%%%%%%%%%%%%%%%%%%%%%%%%%%%%
        \draw[->] (0) to node[pos=.3,auto,swap] {$\scriptstyle w$} (aux);
%%%%%%%%%%%%%%%%%%%%%%%%%%%%%%%%%%%%%%%%%%%%%%%
        \draw (-.25,.65) node[circle, inner sep=1pt, fill=black] () {};
        \draw (.25,-.65) node[circle, inner sep=1pt, fill=black] () {};
        \draw[-] (.5,-1.3) -- (-.5,1.3);
        \draw (-.5,1.3) node[circle, inner sep=0pt, fill=black,
        label={right:{$\scriptstyle L$}}] () {};
      \end{tikzpicture}
  \end{equation*}
More formally, as the diagram illustrates, the two segments in question are the tangents to the boundary $\partial\co(-C_u|H\bigcup C_u|H)$ at the points where a hyperplane (i.e. line) $L$ that separates $C_u|H$ and $-C_u|H$ intersects said boundary.

Denoting by $w$ the orthogonal projection of $v$ on $H$, it follows
from the above remark about the boundary of $\co(-C_u|H\bigcup C_u|H)$
that $w$ must be parallel to the two segments and point away from
$C_u|H$. This determines the direction of $w\in H$ uniquely, and since
$w$ was the orthogonal projection of $v$ on any one of the elements of
an open subset of $\cG(3,2)$, we deduce the desired uniqueness of the
direction of $v$. 
\end{proof}

\begin{remark}\label{re.unique}
   Note that the proof of \Cref{le.unique} actually shows more than the statement claims: it shows that the unique direction in which $C_u$ can be translated without leaving $\co(-C_u\bigcup C_u)$ is that of $\varphi(u)$.
\end{remark}

This gives us, for every $u\in \bS$, a unique unit vector $\varphi(u)=\frac
v{\|v\|}\in \bS_1$ (unit sphere centered at $0$) such that translation
in its direction will position $C_u$ so that it contains $0$. The map
$\varphi$ is in fact very well-behaved. Recall that an {\it odd} map
between two spheres is one that sends antipodes to antipodes. With this in hand, we have

\begin{lemma}\label{le.cont}
  The map $\varphi:\bS\to \bS_1$ deduced from \Cref{le.unique} as in
  the discussion above is continuous and odd.  
\end{lemma}
\begin{proof}
  The fact that $\varphi$ is odd is immediate: if $-\varphi(u)$
  belongs to $C_u$ then $\varphi(u)$ belongs to $-C_u$, meaning that
  translation of $-C_u$ by $-\varphi(u)$ will position the former so
  that it contains $0$. 

As for the continuity of $\varphi$, note first that because of our choice of $\bS$ (such that no $C_u$ contains the origin) there is a positive lower bound on the length of vectors $v$ such that $0\in C_u+v$ as $u$ ranges over $\bS$. In conclusion, there is some $t>0$ such that for all $u\in \bS$ we have 
\begin{equation}\label{eq:cont}
  C_u+t \varphi(u)\in \co(-C_u\bigcup C_u).
\end{equation}
The condition 
\begin{equation*}
  C_u+tv\in \co(-C_u\bigcup C_u)
\end{equation*}
is closed in $(u,v)\in \bS\times \bS_1$, and together with \Cref{re.unique} above, \Cref{eq:cont} shows that the set of pairs $(u,v)$ that satisfy it is precisely the graph of $\varphi:\bS\to \bS_1$. Since we now know that the graph of the map $\varphi$ between compact spaces is closed, we can conclude that it is continuous. 
\end{proof}

\begin{corollary}\label{cor.onto}
  The map $\varphi:\bS\to \bS_1$ from \Cref{le.cont} is onto. 
\end{corollary}
\begin{proof}
  Indeed, \Cref{le.cont} shows that it is a continuous odd map between $2$-spheres. By one
  version of the Borsuk-Ulam theorem (e.g. the main result of \cite{dold}) it follows that the map is not nullhomotopic and hence, because the complement of a point in the $2$-sphere is contractible,  must be onto. 
\end{proof}

\begin{lemma}\label{le.tough}
  For every $u\in \bS$, the orthogonal projection of $C_u$ on $H=\varphi(u)^\perp$ is centrally symmetric. 
\end{lemma}
\begin{proof}
We have to show that the projections of $C_u$ and $-C_u$ on $H$ coincide. If they do not, then there is a point $x$ on the boundary $\partial(C_u|H)$ that does not belong to $-C_u|H$. The line in $\bR^3$ through $x$ and orthogonal to $H$ (and hence parallel to $\varphi(u)$) intersects $-C_u\bigcup C_u$ along a segment (possibly degenerate, i.e. a single point) contained in the boundary $\partial C_u$. 
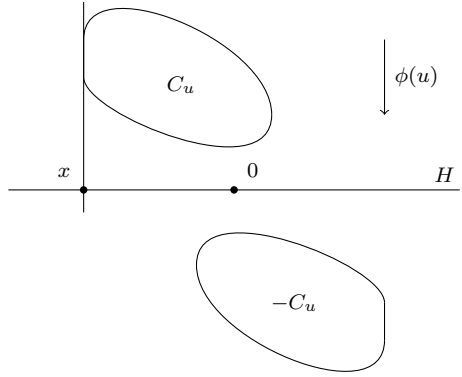
  \begin{equation*}
      \begin{tikzpicture}[auto,baseline=(current  bounding
        box.center),scale=1]
        \node (c) at (-.7,1.4) {$\scriptstyle C_u$};
        \node (-c) at (.8,-1.5) {$\scriptstyle -C_u$};
%        \node (aux) at (-1.125,-.425) {};
        \draw (0,0) node[circle, inner sep=1pt, fill=black, label={45:{$\scriptstyle 0$}}] () {};
%%%%%%%%%%%%%%%%%%%%%%%%%%%%%%%%%%%%%%%%%%%%%%%
        \draw[-] (-2,2) -- (-2,1.5);
        \draw[-] (-2,2) .. controls (-2,3) and (.5,2) .. (.5,1);
        \draw[-] (-2,1.5) .. controls (-2,1) and (.5,0) .. (.5,1);
%%%%%%%%%%%%%%%%%%%%%%%%%%%%%%%%%%%%%%%%%%%%%%%
        \draw[-] (2,-2) -- (2,-1.5);
        \draw[-] (2,-2) .. controls (2,-3) and (-.5,-2) .. (-.5,-1);
        \draw[-] (2,-1.5) .. controls (2,-1) and (-.5,0) .. (-.5,-1);
%%%%%%%%%%%%%%%%%%%%%%%%%%%%%%%%%%%%%%%%%%%%%%%
        \draw[-] (-3,0) -- (3,0);
        \node (H) at (2.8,.2) {$\scriptstyle H$};
%%%%%%%%%%%%%%%%%%%%%%%%%%%%%%%%%%%%%%%%%%%%%%%
        \draw[->] (2,2) to node[pos=.5] {$\scriptstyle \phi(u)$} (2,1);
        \draw[-] (-2,2.5) -- (-2,-.3);
        \draw (-2,0) node[circle, inner sep=1pt, fill=black, label={135:{$\scriptstyle x$}}] () {};
      \end{tikzpicture}
  \end{equation*}
Translation by $\varphi(u)$ will move one of the endpoints of the segment outside of $\co(-C_u\bigcup C_u)$, contradicting 
\begin{equation*}
  C_u+\varphi(u)\subset \co(-C_u\bigcup C_u). 
\end{equation*}
It follows from the contradiction that $-C_u|H=C_u|H$, as desired.
\end{proof}

We are now ready to complete the proof of the main result of this section.

\begin{proof_of_symbis}
  \Cref{le.tough,cor.onto} show that the orthogonal projection of $C$ on every hyperplane in $\bR^3$ has a center of symmetry and hence, according to \Cref{cor.centr}, so does $C$.  
\end{proof_of_symbis}

%%%%%%%%%%%%%%%%%%%%%%%%%%%%%%%%%%%%%%%%%%%%%%%%%%%%%%%%%%%%%%%%%%%%%%%%%%%%%%%%%%%%%%%%%%%%%%%%%%%%%%%%%%%%%%%%%%
%%%%%%%%%%%%%%%%%%%%%%%%%%%%%%%%%%%%%%%%%%%%%%%%%%%%%%%%%%%%%%%%%%%%%%%%%%%%%%%%%%%%%%%%%%%%%%%%%%%%%%%%%%%%%%%%%%
\section{Parkability and finite-dimensional Banach spaces}\label{se.park}

Here, we prove \Cref{th.ell}. Recall the statement:

\begin{theorem}\label{th.ell_bis}
For a convex body $B\subset \bR^n$, $n\ge 3$ the following properties
are equivalent.  
 \begin{enumerate}
 \renewcommand{\labelenumi}{(\roman{enumi})}  
   \item $B$ is an ellipsoid centered at $0$.  
   \item $B$ is centrally symmetric and its intersection with every hyperplane has a center of symmetry (provided it is not empty). 
   \item Every closed convex subset of $B$ is parkable in $B$. 
 \end{enumerate}  
\end{theorem}

We already know from \Cref{pr.ell} that (i) implies (ii), which in
turn implies (iii); hence, it suffices to go backwards. We once again
specialize to $n=3$ in order to simplify some of the proofs and
language.

\begin{proposition}\label{pr.n=3}
  If \Cref{th.ell_bis} holds for $n=3$ then it holds in general. 
\end{proposition}
\begin{proof}
  Property (iii) is clearly preserved by passing to orthogonal
  projections from $\bR^n$ down to its three-dimensional subspaces. If
  the implication (iii) $\Rightarrow$ (i) holds in $\bR^3$, then we
  know that all orthogonal projections of $B$ on such subspaces are
  ellipsoids. This implies that $B$ itself is an ellipsoid by \cite[Theorem 3.1.7]{tomo}. 
\end{proof}

As explained in the introduction, the first priority will be to prove that (iii)
implies the property of being centrally symmetric so as to be able to think of the
convex body as the unit ball of a Banach space structure on
$\bR^n$. The following lemma provides the first step in this direction.

\begin{lemma}\label{le.2impln}
  If the implication (iii) $\Rightarrow$ (centrally symmetric) holds
  in $\bR^2$, then it holds for all $\bR^n$, $n\ge 3$. 
\end{lemma}
\begin{proof}
  Let $B\subset \bR^n$, $n\ge 3$ be a convex body satisfying condition
  (iii) of the theorem, and suppose (iii) implies central symmetry in
  the plane. 

Clearly, (iii) is preserved upon projecting down to any plane, and
hence, by assumption, all projections of $B$ on planes are centrally
symmetric. The conclusion that $B$ itself is then follows from
\Cref{cor.centr} (or rather a variant thereof whereby the center of
symmetry is in fact the origin).  
\end{proof}

\begin{remark}\label{re.hyp}
As seen in \cite[Lemma 29]{berg}, condition (iii) of \Cref{th.ell_bis}
is equivalent to the fact that every hyperplane section of $B$ is
parkable. We will henceforth use this equivalence without further
comment when convenient.   
\end{remark}

\begin{proposition}\label{pr.2}
  Let $B\subset \bR^2$ be a convex body such that every intersection
  of $B$ with a line is parkable in $B$. Then, $B$ is centrally
  symmetric. 
\end{proposition}
\begin{proof}
Let $v\in \bR^2$ be a non-zero vector, and $H_1$ and $H_2$ the two $B$-supporting lines that are orthogonal to $v$. 

{\bf Claim: The convex hull of the segments $H_1\cap B$ and $H_2\cap
  B$ contains the origin.} Indeed, if $p_1\in H_1\cap B$ and $p_2\in
H_2\cap B$ are two points on the boundary of $B$, then by our
hypothesis some translate $\overline{p_1p_2}+w\subset B$ contains the
origin. Because $H_i$ are supporting lines the translates $p_1+w$ and
$p_2+w$ belong to $H_1$ and $H_2$ respectively (i.e. if it is
non-zero, then $w$ is parallel to $H_1$ and $H_2$); this finishes the proof of the claim.  
  
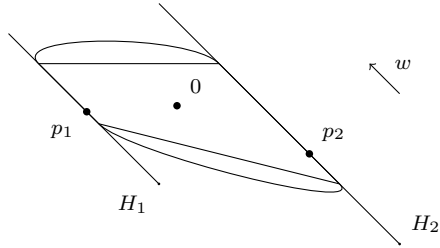
\begin{equation*}
      \begin{tikzpicture}[auto,baseline=(current  bounding
        box.center),scale=.8]
        \draw (0,-1) node[circle, inner sep=0pt, fill=black, label={225:{$\scriptstyle H_1$}}] () {};
        \draw (-1.2,.2) node[circle, inner sep=1pt, fill=black, label={225:{$\scriptstyle p_1$}}] (p1) {};
        \draw (4,-2) node[circle, inner sep=0pt, fill=black, label={45:{$\scriptstyle H_2$}}] () {};
        \draw (2.5,-.5) node[circle, inner sep=1pt, fill=black, label={45:{$\scriptstyle p_2$}}] (p2) {};
        \draw (.3,.3) node[circle, inner sep=1pt, fill=black, label={45:{$\scriptstyle 0$}}] () {};
%%%%%%%%%%%%%%%%%%%%%%%%%%%%%%%%%%%%%%%%%%%%%%%
        \draw[-] (-2,1) -- (-1,0);
        \draw[-] (1,1) -- (3,-1);
        \draw[-] (-2,1) -- (1,1);
        \draw[-] (-1,0) -- (3,-1);
        \draw[-] (-2,1) .. controls (-2.5,1.5) and (.5,1.5) .. (1,1);
        \draw[-] (-1,0) .. controls (-.5,-.5) and (3.5,-1.5) .. (3,-1);
%%%%%%%%%%%%%%%%%%%%%%%%%%%%%%%%%%%%%%%%%%%%%%%
        \draw[-] (-2.5,1.5) -- (0,-1);
        \draw[-] (0,2) -- (4,-2);
%%%%%%%%%%%%%%%%%%%%%%%%%%%%%%%%%%%%%%%%%%%%%%%
        \draw[->] (4,.5) to node[pos=.5,auto,swap] {$\scriptstyle w$} (3.5,1);
      \end{tikzpicture}
  \end{equation*}
Given the claim, the conclusion now follows from the technical \Cref{le.supports} below. 
\end{proof}

\begin{lemma}\label{le.supports}
Let $B\subset \bR^2$ be a convex body. Suppose that for every non-zero $v\in \bR^2$ the two supporting lines $L_1=L_1(v)$ and $L_2=L_2(v)$ of $B$ that are orthogonal to $v$ satisfy
\begin{equation*}
  0\in \co((L_1\cap B)\cup(L_2\cap B)). 
\end{equation*}
Then, $B$ is centrally symmetric. 
\end{lemma}
\begin{proof}
We have to show that $B$ and $-B$ coincide. Assuming they do not, the condition on $B$ at least ensures that the union $-B\cup B$ is convex. Indeed, otherwise for some $p\in \partial B\cap \partial(-B)$ some $B$-supporting line $L$ through $p$ would not be $-B$-supporting, and hence $v\perp L$ would violate the hypothesis.

In conclusion, $B$ and $-B$ admit common parallel supporting lines $L$ and $-L$ at two points $p$ and $-p$ respectively in the intersection $\partial B\cap \partial(-B)$. We will moreover choose coordinates in $\bR^2$ so that the segment $\overline{(-p)p}$ is horizontal and $L$ is vertical. We will derive a contradiction from the assumption that the portions of $\partial B$ and $\partial(-B)$ lying in the lower half plane are distinct. 
  
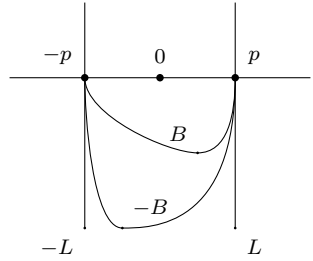
\begin{equation*}
      \begin{tikzpicture}[auto,baseline=(current  bounding
        box.center),scale=1]
        \draw (0,0) node[circle, inner sep=1pt, fill=black, label={90:{$\scriptstyle 0$}}] () {};
        \draw (-1,0) node[circle, inner sep=1pt, fill=black, label={135:{$\scriptstyle -p$}}] () {};
        \draw (1,0) node[circle, inner sep=1pt, fill=black, label={45:{$\scriptstyle p$}}] () {};
        \draw (-1,-2) node[circle, inner sep=0pt, fill=black, label={225:{$\scriptstyle -L$}}] () {};
        \draw (1,-2) node[circle, inner sep=0pt, fill=black, label={315:{$\scriptstyle L$}}] () {};
        \draw (.5,-1) node[circle, inner sep=0pt, fill=black, label={100:{$\scriptstyle B$}}] () {};
        \draw (-.5,-2) node[circle, inner sep=0pt, fill=black, label={85:{$\scriptstyle -B$}}] () {};
%%%%%%%%%%%%%%%%%%%%%%%%%%%%%%%%%%%%%%%%%%%%%%%
        \draw[-] (-1,1) -- (-1,-2);
        \draw[-] (1,1) -- (1,-2);
        \draw[-] (-2,0) -- (2,0);
        \draw[-] (-1,0) .. controls (-1,-.5) and (.1,-1) .. (.5,-1);
        \draw[-] (.5,-1) .. controls (.9,-1) and (1,-.5) .. (1,0);
        \draw[-] (-1,0) .. controls (-1,-.5) and (-.9,-2) .. (-.5,-2);
        \draw[-] (-.5,-2) .. controls (-.1,-2) and (1,-2) .. (1,0);
      \end{tikzpicture}
  \end{equation*}
We can now identify $\overline{(-p)p}$ with a closed interval $[-a,a]$, $a>0$ and the portions of $\partial B$ and $\partial(-B)$ depicted above with graphs of convex functions $f<0$ and $g<0$ respectively defined on $(-a,a)$. Moreover, the assumption $B\ne -B$ translates without loss of generality to $g(x)<f(x)$ for all $x\in (-a,a)$. 

In this setup, the meaning of the hypothesis is that whenever a line through the origin intersects the graphs of $f$ and $g$ at smooth points $(x,f(x))$ and $(y,f(y))$ respectively, the tangents to the two graphs at those points are parallel; equivalently, $f'(x)=g'(y)$. 

If, in the previous paragraph, we choose $x$ to be negative, then $y<x$. Because $g$ is convex, its derivative is non-decreasing, and hence
\begin{equation*}
  f'(x)=g'(y)\le g'(x)
\end{equation*}
 (always assuming we are working with points where the derivatives exist, which is the case for all but countably many of the elements in the domain $(-a,a)$).

In conclusion, whenever both $f$ and $g$ are differentiable at $x\in (-a,0]$ we have $f'(x)\le g'(x)$. This, together with
\begin{equation*}
  f(0)=\int_{-a}^0f'(x)\ \mathrm{d}x\le g(0)=\int_{-a}^0g'(x)\ \mathrm{d}x
\end{equation*}
contradicts our assumption that $g(0)<f(0)$. 
\end{proof}

Using \Cref{pr.2}, we obtain

\begin{corollary}\label{cor.iii_implies_sym}
  Let $B\subset \bR^n$, $n\ge 3$ be a convex body satisfying condition
  (iii) of \Cref{th.ell_bis}. Then, $B$ is centrally symmetric.  
\end{corollary}
\begin{proof}
  This follows immediately from \Cref{le.2impln} via \Cref{pr.2}.
\end{proof}

We now know that all convex bodies satisfying the weakest condition (iii) of \Cref{th.ell_bis} are centered at the origin, and as a consequence we henceforth specialize the discussion to such bodies. Next, our goal will be to prove an $n=3$ version of \Cref{pr.aux}, announced above.

\begin{proposition}\label{pr.aux_bis}
  Let $B$ be a convex centrally symmetric body $B\subset \bR^3$ satisfying condition (iii) from \Cref{pr.ell}. Then, for every linear hyperplane $H\subset \bR^3$, the non-empty intersections
  \begin{equation*}
    (H+x)\cap B,\quad x\in \bR^3
  \end{equation*}
have centers of symmetry. Moreover, these centers are collinear.
\end{proposition}

Before embarking on a proof, we need some preparation. As explained in the discussion following \cite[Question 31]{berg}, convex bodies $B$ with parkable hyperplane sections have the following property: for any linear hyperplane $H\subset \bR^3$ there is a line $L\subset \bR^3$ such that $B$ has a supporting translate of $L$ at every point of $H\cap \partial B$ (here, `supporting' is the natural extension to lines of the term `supporting hyperplane'; it simply means that the line intersects only the boundary of $B$).

It is also explained in loc. cit. how this property is in a sense dual to one considered by Blaschke \cite[pp. 157-159]{bla} in relation to illuminating convex bodies. In this latter reference, (smooth) convex bodies are considered with the property that when illuminated with parallel light rays, the boundary curve of the illuminated region is planar. In other words, given a line $L$, the intersections of the $B$-supporting translates of $L$ with $B$ form a planar curve in $\partial B$. This is the precise opposite of the situation in the previous paragraph.

Motivated by these considerations, we label the two dual properties described above.

\begin{definition}\label{def.bla}
  Let $B\subset \bR^3$ be a convex body containing $0$ in its interior. 

$B$ has {\it the Blaschke property} (or {\it is Blaschke}, for short) if for any line $L$ there exists a linear hyperplane $H\subset \bR^3$ such that all points in $H\cap \partial B$ lie on some $B$-supporting translate of $L$. 

$B$ has {\it the dual Blaschke property} (or {\it is dual Blaschke}) if for any linear hyperplane $H\subset \bR^3$ there is a line $L\in \bR^3$ such that all points in $H\cap \partial B$ lie on some $B$-supporting translate of $L$.  
\end{definition}

With these terms in place, the discussion preceding the definition amounts to (\cite[p. 257]{berg})

\begin{proposition}\label{pr.co_bla}
  A centrally symmetric convex body in $\bR^3$ satisfying condition (iii) of \Cref{th.ell_bis} is dual Blaschke. 
\qedhere
\end{proposition}

Our next aim is to prove the dual version of this result.

\begin{proposition}\label{pr.bla}
  A centrally symmetric convex body in $\bR^3$ satisfying condition (iii) of \Cref{th.ell_bis} is Blaschke.   
\end{proposition}

We now introduce some tools necessary in the proof of \Cref{pr.bla}. Consider a convex body $B$ as in \Cref{pr.co_bla}. For every unit vector $v\in \bR^3$, denote by $H_v$ the plane orthogonal to $v$. We define the subset $\psi(v)=\psi_B(v)$ of the unit sphere to consist of all unit vectors $w$ such that
\begin{equation*}
  w\cdot v>0 \text{ and the line } \{p+tw\ |\ t\in \bR\} \text{ is }B-\text{supporting for every }p\in \partial B\cap H_v. 
\end{equation*}
In other words, it is the set of unit vectors that make acute angles with $v$ and pointing along the directions of those lines $L$ associated to the plane $H_v$ as in \Cref{pr.co_bla,def.bla}. 

\Cref{pr.co_bla} is what makes $\psi(v)$ non-empty. Note moreover that since the lines supporting $B$ at points in $\partial B\cap H$ cannot be contained in $H$ (for any plane $H$), $\psi(v)$ is closed in the unit sphere $\bS^2$. Finally, it is also {\it convex} as a subset of the sphere, in the sense that for $w,w'\in \psi(v)$ the rescaled convex combinations
\begin{equation*}
  \frac{tw+(1-t)w'}{\|tw+(1-t)w'\|},\ t\in (0,1)
\end{equation*}
all belong to $\psi(v)$. All in all, we have defined a map 
\begin{equation*}
  \psi=\psi_B:\bS^2\to\text{ closed convex subsets of }\bS^2. 
\end{equation*}
The map $\psi$ has a number of other properties that are easy to check:

\begin{lemma}\label{le.odd}
  For a convex body $B$ as in \Cref{pr.co_bla}, the map $\psi=\psi_B$ is odd in the sense that 
  \begin{equation*}
    \psi(-v)=-\psi(v)
  \end{equation*}
and upper semicontinuous in the sense that its graph
\begin{equation*}
  \{(v,w)\in \bS^2\times \bS^2\ |\ w\in \psi(v)\}
\end{equation*}
is closed in $\bS^2\times \bS^2$. 
\end{lemma}
\begin{proof}
  The property of being odd is immediate from the definition of $\psi$: $v$ and $-v$ define the same plane $H_v=H_{-v}$, and hence the same set of lines supporting $B$ at the points of $\partial B\cap H_v$. The upper semicontinuity claim is similarly routine, using the observation that the set of $B$-supporting affine lines is closed in the set of all affine lines of $\bR^3$. 
\end{proof}

\begin{lemma}\label{le.psi_onto}
  An odd, upper semicontinuous map 
  \begin{equation*}
       \psi=:\bS^2\to\text{ closed convex subsets of }\bS^2
  \end{equation*}
is onto in the sense that every $w\in \bS^2$ belongs to some set $\psi(v)$, $v\in \bS^2$. 
\end{lemma}
\begin{proof}
  Assume the contrary, with, say, $w\in \bS^2$ lying outside all of the sets $\psi(v)$ as $v$ ranges over $\bS^2$. Flowing along great circles through $w$ down to the equator in the plane orthogonal to $w$, we can deform $\psi$ into an odd, upper semicontinuous map
  \begin{equation*}
    \psi=:\bS^2\to\text{ closed convex subsets of }\bS^1. 
  \end{equation*}
The existence of such a map contradicts the multivalued Borsuk-Ulam theorem as stated e.g. in \cite[Corollary 2.4]{multi_BU}.
\end{proof}

We have now effectively proven \Cref{pr.bla} above.

\begin{proof_of_bla}
Using the notation introduced above, the statement amounts to showing that $\psi_B$ is onto; this is precisely what \Cref{le.psi_onto} does. 
\end{proof_of_bla}

This provides us with the necessary tools to attack \Cref{pr.aux_bis}.

\begin{proof_of_aux_bis}
  The statement stipulates a condition that is closed among planes in $\cG(3,2)$, so it suffices to prove that this condition holds for a dense set of planes. Specifically, we will prove it for planes $H$ with the following property:
 
{\it The two supporting planes of $B$ that are parallel to $H$ each intersect $B$ at a single point.}

Now fix such a plane $H$, and let $p$ and $q=-p$ be the two points where the two $B$-supporting translates of $H$ intersect $B$. Let also $K$ be a planar section $(H+x)\cap B$ with non-empty relative interior. Finally, fix a line $L\in H$. 

In the affine plane $H+x$, the convex body $K$ has two supporting lines parallel to $L$; denote these by $L_1$ and $L_2$.

{\bf Claim: The convex hull of $L_i\cap K$ contains the point $r=(H+x)\cap \overline{pq}$.} Given the claim, we can finish the proof of the proposition as follows. 

Since the claim holds for any line $L\in H$, we can apply \Cref{le.supports} inside the affine plane $H+x$, with the point $r$ regarded as the origin, to conclude that $K$ has this point as its center of symmetry. In conclusion, we obtained the desired result: all planar sections of $B$ that are parallel to $H$ are centrally symmetric along points on the line containing $p$ and $q$.

It thus remains to prove the claim; this goal will take up the rest of the proof.

According to \Cref{pr.bla}, there is a plane $H'\subset \bR^3$ with the property that every point in $H'\cap \partial B$ is contained in a $B$-supported line parallel to $L$. By our choice of $H$ (so that its two $B$-supporting translates meet $B$ at single points $p$ and $q$ respectively), we have $p,q\in H'$. As a consequence of this, the endpoints of the segment $H'\cap K$ are contained in the $B$-supporting translates $L_1$ and $L_2$; this concludes the proof. 
\end{proof_of_aux_bis}

\begin{remark}\label{re.32to31}
  Note that given the plane $H\subset \bR^3$, the line containing the centers of symmetry for the non-empty planar sections $(H+x)\cap B$ is uniquely determined; this provides a map from the Grassmannian of planes $\cG(3,2)$ to the Grassmannian $\cG(3,1)$ of lines in $\bR^3$. 
\end{remark}

We now take the first step towards constructing a map $\cG(3,1)\to \cG(3,2)$ in the opposite direction to that of \Cref{re.32to31}.

\begin{lemma}\label{le.aux_bis_dual}
  Suppose we are under the assumptions of \Cref{pr.aux_bis}, and $L$ is the line containing the centers of symmetry of $(H+x)\cap B$ for $x$ ranging over $\bR^3$. Then, the midpoint of every non-empty intersection $(L+y)\cap B$ belongs to $H$. 
\end{lemma}
\begin{proof}
  \Cref{pr.aux_bis} can be restated as saying that the linear involution $T$ of $\bR^3$ that acts as the identity along $L$ and as $x\mapsto -x$ on $H$ preserves the convex body $B$. The involution $-T$ also preserves $B$, acts as the identity on $H$, and as $x\mapsto -x$ along $L$. In other words, $-T$ preserves $H$ and reverses every segment $(L+y)\cap B$; this implies the conclusion. 
\end{proof}

We can now prove the dual version of \Cref{pr.aux_bis}.

\begin{proposition}\label{pr.aux_bis_dual}
  Suppose we are under the hypotheses of \Cref{pr.aux_bis}. For any line $L\subset \bR^3$, the midpoints of the non-empty intersections $(L+y)\cap B$ are coplanar. 
\end{proposition}
\begin{proof}
  \Cref{le.aux_bis_dual} above already proves the statement for those lines $L$ arising as images of planes $H\in \cG(3,2)$ through the map $\cG(3,2)\to \cG(3,1)$ from \Cref{re.32to31}. It is thus sufficient to show that this map is onto. In other words, we have to prove that {\it every} line $L$ contains the symmetry centers of the non-empty planar sections $(H+x)\cap B$ for some $H\in \cG(3,2)$.

Consider the map $\phi=\phi_B:\bS^2\to \bS^2$ defined as follows. For a unit vector $v\in \bS^2$, let $H=H_v$ be the plane orthogonal to $v$, and take $\phi(v)$ be the unit vector making an acute angle with $v$ and pointing along the line $L$ that contains the centers of symmetry of $(H+x)\cap B$, $x\in \bR^3$. 

Clearly, $\phi$ is continuous and odd, in the sense that $\phi(-v)=-\phi(v)$. We can now apply \Cref{le.psi_onto} in the simpler case of ordinary (rather than multivalued) maps to conclude via Borsuk-Ulam that $\phi$ is onto. 
\end{proof}

\Cref{pr.aux_bis,pr.aux_bis_dual} allow us to introduce the following

\begin{notation}\label{not.prime}
  For a linear subspace $H\subset \bR^3$ define $H'$ to be
  \begin{itemize}
    \item the zero subspace $\{0\}$ if $H=\bR^3$;
    \item the line containing the centers of symmetry of $(H+x)\cap B$ if $H$ is a plane;
    \item the plane containing the midpoints of $(H+y)\cap B$ if $H$ is a line;
    \item $\bR^3$ if $H=\{0\}$. 
  \end{itemize}
\end{notation}

The operation $H\mapsto H'$ on the subspace of $\bR^3$ has the following properties.

\begin{lemma}\label{le.inv}
  \begin{enumerate}
  \renewcommand{\labelenumi}{(\arabic{enumi})}  
    \item For any two subspaces $L,H$ of $\bR^3$ we have $L\subseteq H\Rightarrow L'\supseteq H'$; 
    \item For every subspace $H\subseteq \bR^3$ we have $H''=H$; 
    \item For every subspace $H\subseteq \bR^3$ we have $H\cap H'=\{0\}$. 
  \end{enumerate}
\end{lemma}
\begin{proof}
  In all three statements, the interesting cases are those where the subspaces in question are 1- or 2-dimensional. For this reason, we treat only these cases in the proof. 

{\bf (1)} Here, we may as well assume that $L$ is a line contained in the plane $H$. $H'$ is by definition the line containing the symmetry centers of the sections $(H+x)\cap B$. Every such center is the midpoint of the segment $(L+x)\cap B$, and is thus on $L'$ by the definition of the latter.

{\bf (2)} The identity $H=H''$ for planes is simply a restatement of \Cref{le.aux_bis_dual} above. Applying the operation $\bullet\mapsto \bullet'$ once more we get $H'=H'''$ for planes $H$, and hence $L=L''$ for lines $L$ follows from the observation that every line arises as $H'$ for some plane $H$ (this latter claim is essentially the surjectivity of the map $\phi_B$ from the proof of \Cref{pr.aux_bis_dual}).

{\bf (3)} This is immediate: when $H$ is a plane, for instance, there are non-empty sections $(H+x)\cap B$ for non-zero $x\in \bR^3$, and hence there are points of the line $H'$ that are not contained in $H$. A similar argument applies when $H$ is a line (or alternatively, by part (2), in that case $H=H''$ and we can apply the previous argument to the plane $H'$).   
\end{proof}

Properties (1), (2) and (3) in \Cref{le.inv} are, according to \cite[Theorem 1]{KM44}, precisely what is necessary in order to ensure that there is a Hilbert space structure on $\bR^3$ for which $H\mapsto H'$ is the orthogonal complement operation. Such a Hilbert space structure is given by an inner product of the form 
\begin{equation*}
  \langle v,w\rangle = v\cdot Aw
\end{equation*}
where $A\in M_3=M_3(\bR)$ is a positive operator and $\cdot$ is the usual dot product. Since the image of an ellipsoid through a positive operator $A:\bR^3\to \bR^3$ is again an ellipsoid, we may as well assume (for the purposes of proving \Cref{th.ell_bis}) that $A$ is the identity matrix and hence $K\mapsto K'$ is the usual orthogonal complement operation in $\bR^3$; we will do this throughout the rest of the section, in order to simplify the discussion.

Now consider the usual transposition operation $T\mapsto T^t$ on $M_3$. It satisfies conditions (1), (2) and (3) of \Cref{th.istr} above. Our goal will be to apply that result to the Banach space $(\bR^3,\|\cdot\|_B)$ induced by the centrally symmetric convex body $B$ as explained in \Cref{subse.ban}. In preparation for that, note that the symmetric idempotents $P\in M_3$ (i.e. those satisfying $P^2=P=P^t$) are precisely those whose range and kernel are orthogonal.

\begin{lemma}\label{le.istr_applies}
  If $\|\cdot\|$ denotes the norm induced on $M_3$ by $\|\cdot\|_B$, then $\|P\|^2\le 1$ for all symmetric idempotents $P\in M_3$
\end{lemma}
\begin{proof}
  If the range of $P$ is three- or zero-dimensional then $P$ is the identity or the zero operator respecticely, so there is nothing to prove. It thus remains to prove the claim when $\dim(\mathrm{Im}(P))$ is $2$ or $1$. We tackle the former case; the latter is entirely analogous. 

If $H\in \cG(3,2)$ is the range of $P$, then $P$ is the orthogonal projection on $H$. Since, as explained in the discussion following \Cref{le.inv}, we are assuming that $H'=H^\perp$, the convex body $B$ (i.e. the unit ball of $(\bR^3,\|\cdot\|_B)$) is contained in the orthogonal cylinder based on $H\cap B$. But this means that if $\|x\|_B=1$ (i.e. $x\in \partial B$) then the orthogonal projection $Px$ on $H$ is contained in $B\cap H$, and hence $\|Px\|_B\le 1$. In conclusion, 
\begin{equation*}
  \|P\| = \sup_{\|x\|_B\le 1}\|Px\|_B\le 1. 
\end{equation*}
This finishes the proof. 
\end{proof}

This concludes the preparatory material necessary for the proof of the main result of this section.

\begin{proof_of_ellbis}
  As mentioned after the statement of the theorem, it suffices to prove that (iii) implies that $B$ is an ellipsoid. Moreover, \Cref{pr.n=3} shows that it is enough to work with $n=3$.

\Cref{cor.iii_implies_sym} ensues that $B$ is centrally symmetric, and hence can be regarded as the unit ball of a Banach space $(\bR^3,\|\cdot\|_B)$. \Cref{le.istr_applies,th.istr} now ensure that the norm $\|\cdot\|_B$ is induced by an inner product on $\bR^3$, and hence the unit ball of $\|\cdot\|_B$ is an ellipsoid, as explained in \Cref{subse.ban}. 
\end{proof_of_ellbis}

%%%%%%%%%%%%%%%%%%%%%%%%%%%%%%%%%%%%%%%%%%%%%%%%%%%%%%%%%%%%%%%%%%%%%%%%%%%%%%%%%%%%%%%%%%%%%%%%%%%%%%%%%%%%%%%%%%
%%%%%%%%%%%%%%%%%%%%%%%%%%%%%%%%%%%%%%%%%%%%%%%%%%%%%%%%%%%%%%%%%%%%%%%%%%%%%%%%%%%%%%%%%%%%%%%%%%%%%%%%%%%%%%%%%%
\section{Illuminated bodies and subspace lattice involutions}\label{se.illum}

The techniques employed in \Cref{se.park} will also help in improving on a characterization of ellipsoids via illumination by parallel rays due to Blaschke (\cite[pp. 157-159]{bla}). For the purpose of stating the result briefly, we introduce one more piece of terminology.

\begin{definition}\label{def.wbla}
  A convex body $B$ in $\bR^3$ has {\it the weak Blaschke property} (or {\it is weakly Blaschke}) if for any line $L$ there exists an affine plane $H\subset \bR^3$ such that $H$ intersects the interior of $B$ and all points in $H\cap \partial B$ lie on some $B$-supporting translate of $L$. 

$B$ has {\it the weak dual Blaschke property} (or {\it is weakly dual Blaschke}) if for any linear plane $H\subset \bR^3$ there is a translate $H'$ of $H$ and a line $L\in \bR^3$ such that $H'$ intersects the interior of $B$ and all points in $H'\cap \partial B$ lie on some $B$-supporting translate of $L$.  
\end{definition}

\begin{remark}\label{re.wbla}
  Note the difference to \Cref{def.bla}: in \Cref{def.wbla} we do not require the planes to pass through a given, fixed point (such as the origin in the case of the former definition). This justifies the adjective `weak' in the definition. In fact, it is easy to see that centrally symmetric weakly (dual) Blaschke convex bodies are automatically (dual) Blaschke.  
\end{remark}

\begin{remark}
The weak Blaschke property can be interpreted as saying that if the body is illuminated with parallel rays, then the boundary of the shaded region contains a coplanar curve; hence the title of this section. 
\end{remark}

With this language in place, Blaschke's result referred to above states that a smooth, strictly convex convex body in $\bR^3$ that is weakly Blaschke must be an ellipsoid. In this section we remove the smoothness and strict convexity requirements:

\begin{theorem}\label{th.illum}
  A convex $B\subset \bR^3$ that is weakly Blaschke is an ellipsoid. 
\end{theorem}

We once more prepare the ground before giving the proof proper. First, we show that the weak Blaschke property entails its dual.

\begin{lemma}\label{le.self-dual}
  A weakly Blaschke convex body is weakly dual Blaschke. 
\end{lemma}
\begin{proof}
  This is very similar in spirit to the proof of \Cref{pr.bla} from \Cref{pr.co_bla} via \Cref{le.psi_onto}, using the same multivalued Borsuk-Ulam theorem. 

To any unit vector $v\in \bS^2$ associate the subset of the unit $2$-sphere consisting of those vectors $w$ with the property that $w\cdot v>0$ and $w$ is orthogonal to those planes associated as in the definition of the weak Blaschke property to the line containing $v$. 

This gives rise, as in the discussion following \Cref{pr.bla}, to a map 
\begin{equation*}
       \psi=:\bS^2\to\text{ closed convex subsets of }\bS^2.
\end{equation*}
$\psi$ can be shown to be upper semicontinuous and odd as in \Cref{le.odd}, and hence is onto according to \Cref{le.psi_onto}. This concludes the proof: we have just shown that {\it every} plane in $\bR^3$ is parallel to some affine plane associated to a line as in the definition of the weak Blaschke property. 
\end{proof}

Next, we prove the existence of a center of symmetry.

\begin{lemma}\label{le.wbla_symm}
  A weakly Blaschke convex body $B\subset \bR^3$ has a center of symmetry. 
\end{lemma}
\begin{proof}
  It suffices to show that the group of affine symmetries of $B$ is infinite. For this purpose, let $H\subset \bR^3$ be a plane  with the property that the two $B$-supporting translates of $H$ intersect $B$ at single point $p$ and $q$.  

We now proceed very much as in the proof of \Cref{pr.aux_bis}. We denote by $K$ a planar section $(H+x)\cap B$ of $B$ that has non-empty relative interior. Then choose a line $L\subset H$. Just as in the cited proof, we can show using the weak Blaschke property that if $L_1$ and $L_2$ are the $K$-supporting translates of $L$, then the convex hull of $L_i\cap K$ contains the point $\overline{pq}\cap K$.

\Cref{le.supports} then ensures that $K$ has a center of symmetry at $\overline{pq}\cap K$. Since the $H$-parallel section $K$ of $B$ was arbitrary, we conclude that the affine involution of $\bR^3$ that preserves the points on the line $pq$ and acts as reflection across $pq\cap (H+x)$ along the plane $H+x$ preserves $B$. 
 
The above construction provides us with infinitely many affine symmetries of $B$, one for each choice of plane $H$ with the generic property specified at the beginning. 
\end{proof}

\begin{proof_of_illum}
  We now know from \Cref{le.self-dual} that $B$ is both weakly Blaschke and weakly dual Blaschke. Since moreover it has a center of symmetry by \Cref{le.wbla_symm}, we can assume that $B$ is centered at $0$ and hence is both Blaschke and dual Blaschke (cf. \Cref{re.wbla}). Running through the proof of \Cref{th.ell_bis}, this is sufficient to conclude that $B$ is an ellipsoid. 
\end{proof_of_illum}

%%%%%%%%%%%%%%%%%%%%%%%%%%%%%%%%%%%%%%%%%%%%%%%%%%%%%%%%%%%%%%%%%%%%%%%%%%%%%%%%%%%%%%%%%%%%%%%%%%%%%%%%%%%%%%%%%%
%%%%%%%%%%%%%%%%%%%%%%%%%%%%%%%%%%%%%%%%%%%%%%%%%%%%%%%%%%%%%%%%%%%%%%%%%%%%%%%%%%%%%%%%%%%%%%%%%%%%%%%%%%%%%%%%%%

\bibliography{park}{}

\begin{thebibliography}{10}

\bibitem{berg}
George~M. Bergman.
\newblock Mapping radii of metric spaces.
\newblock {\em Pacific J. Math.}, 236(2):223--261, 2008.

\bibitem{bla}
Wilhelm Blaschke.
\newblock {\em Kreis und {K}ugel}.
\newblock Chelsea Publishing Co., New York, 1949.

\bibitem{con}
John~B. Conway.
\newblock {\em A course in functional analysis}, volume~96 of {\em Graduate
  Texts in Mathematics}.
\newblock Springer-Verlag, New York, second edition, 1990.

\bibitem{dold}
Albrecht Dold.
\newblock Simple proofs of some {B}orsuk-{U}lam results.
\newblock In {\em Proceedings of the {N}orthwestern {H}omotopy {T}heory
  {C}onference ({E}vanston, {I}ll., 1982)}, volume~19 of {\em Contemp. Math.},
  pages 65--69. Amer. Math. Soc., Providence, RI, 1983.

\bibitem{tomo}
Richard~J. Gardner.
\newblock {\em Geometric tomography}, volume~58 of {\em Encyclopedia of
  Mathematics and its Applications}.
\newblock Cambridge University Press, Cambridge, second edition, 2006.

\bibitem{istr}
Vasile~I. Istr{\u{a}}{\c{t}}escu.
\newblock An algebraic characterization of complete inner product spaces.
\newblock {\em Internat. J. Math. Math. Sci.}, 9(1):47--53, 1986.

\bibitem{multi_BU}
Marek Izydorek and Jan Jaworowski.
\newblock Parametrized {B}orsuk-{U}lam theorems for multivalued maps.
\newblock {\em Proc. Amer. Math. Soc.}, 116(1):273--278, 1992.

\bibitem{KM44}
Shizuo Kakutani and George~W. Mackey.
\newblock Two characterizations of real {H}ilbert space.
\newblock {\em Ann. of Math. (2)}, 45:50--58, 1944.

\bibitem{rud}
Walter Rudin.
\newblock {\em Functional analysis}.
\newblock International Series in Pure and Applied Mathematics. McGraw-Hill,
  Inc., New York, second edition, 1991.

\bibitem{Rya15}
Dmitry Ryabogin.
\newblock A {L}emma of {N}akajima and {S}\"uss on convex bodies.
\newblock {\em Amer. Math. Monthly}, 122(9):890--892, 2015.

\end{thebibliography}
\bibliographystyle{plain}
\addcontentsline{toc}{section}{References}

\end{document}